\documentclass[10pt,reqno]{amsart}

\usepackage{amsthm, mathrsfs, amsmath, amstext, amsxtra, amsfonts, dsfont, amssymb, color}
\usepackage{lmodern}
\definecolor{green_dark}{rgb}{0,0.6,0}
\usepackage[colorlinks, linkcolor=red, citecolor=blue, urlcolor=green_dark, pagebackref, hypertexnames=false]{hyperref}

\newcommand{\N}{\mathbb N}
\newcommand{\Z}{\mathbb Z}

\newcommand{\R}{\mathbb R}
\newcommand{\C}{\mathbb C}

\newcommand{\re}[1]{\mbox{Re} \ #1} 
\newcommand{\im}[1]{\mbox{Im} \ #1} 
\newcommand{\scal}[1]{\left\langle #1 \right\rangle} 

\newcommand{\defendproof}{\hfill $\Box$} 

\newtheorem{theorem}{Theorem}[section]
\newtheorem{defi}[theorem]{Definition}
\newtheorem{lem}[theorem]{Lemma} 
\newtheorem{prop}[theorem]{Proposition}
\newtheorem{coro}[theorem]{Corollary} 
\theoremstyle{definition}

\title[Global existence mass-critical fourth-order Schr\"odinger]{Global existence for the defocusing mass-critical nonlinear fourth-order Schr\"odinger equation below the energy space} 

\author[V. D. Dinh]{Van Duong Dinh}
\address[V. D. Dinh]{Institut de Math\'ematiques de Toulouse, Universit\'e Toulouse III Paul Sabatier, 31062 Toulouse Cedex 9, France}
\email{dinhvan.duong@math.univ-toulouse.fr}


\begin{document}

\maketitle
\begin{abstract}
In this paper, we consider the defocusing mass-critical nonlinear fourth-order Schr\"odinger equation. Using the $I$-method combined with the interaction Morawetz estimate, we prove that the problem is globally well-posed in $H^\gamma(\R^d), 5\leq d \leq 7$ with $\gamma(d)<\gamma<2$, where $\gamma(5)=\frac{8}{5}, \gamma(6)=\frac{5}{3}$ and $\gamma(7)=\frac{13}{7}$.    
\end{abstract}


\section{Introduction}
\setcounter{equation}{0}
Consider the defocusing mass-critical nonlinear fourth-order Schr\"odinger equation, namely
\begin{align}
\left\{
\begin{array}{rcl}
i\partial_t u(t,x) + \Delta^2 u(t,x) &=& -(|u|^{\frac{8}{d}} u)(t,x), \quad t\geq 0, x \in \R^d, \\
u(0,x) &=& u_0(x) \in H^\gamma(\R^d), 
\end{array}
\right.
\tag{NL4S}
\end{align}
where $u(t,x)$ is a complex valued function in $\R^+ \times \R^d$. \newline
\indent The fourth-order Schr\"odinger equation was introduced by Karpman \cite{Karpman} and Karpman-Shagalov \cite{KarpmanShagalov} taking into account the role of small fourth-order dispersion terms in the propagation of intense laser beams in a bulk medium with Kerr nonlinearity. The study of nonlinear fourth-order Schr\"odinger equation has attracted a lot of interest in the past several years (see \cite{Pausader}, \cite{Pausadercubic}, \cite{HaoHsiaoWang06}, \cite{HaoHsiaoWang07}, \cite{HuoJia}, \cite{MiaoXuZhao09}, \cite{MiaoXuZhao11}, \cite{MiaoWuZhang} and references therein).\newline
\indent It is known (see \cite{Dinhfract} or \cite{Dinhfourt}) that (NL4S) is locally well-posed in $H^\gamma(\R^d)$ for $\gamma>0$ satisfying for $d\ne 1, 2, 4$,
\begin{align}
\lceil \gamma \rceil \leq 1+\frac{8}{d}. \label{regularity condition}
\end{align} 
Here $\lceil \gamma \rceil$ is the smallest integer greater than or equal to $\gamma$. This condition ensures the nonlinearity to have enough regularity. The time of existence depends only on the $H^\gamma$-norm of initial data. Moreover, the local solution enjoys mass conservation, i.e.
\[
M(u(t)):=\|u(t)\|^2_{L^2(\R^d)} = \|u_0\|^2_{L^2(\R^d)},
\] 
and $H^2$-solution has conserved energy, i.e.
\[
E(u(t)):=\int_{\R^d} \frac{1}{2}|\Delta u(t,x)|^2 + \frac{d}{2d+8}|u(t,x)|^{\frac{2d+8}{d}} dx = E(u_0).
\]
The conservations of mass and energy together with the persistence of regularity (see \cite{Dinhfourt}) yield the global well-posedness for (NL4S) in $H^\gamma(\R^d)$ with $\gamma\geq 2$ satisfying for $d\ne 1, 2, 4$, $(\ref{regularity condition})$. We also have (see \cite{Dinhfract} or \cite{Dinhfourt}) the local well-posedness for (NL4S) with initial data $u_0 \in L^2(\R^d)$ but the time of existence depends on the profile of $u_0$ instead of its $H^\gamma$-norm. The global existence holds for small $L^2$-norm initial data. For large $L^2$-norm initial data, the conservation of mass does not immediately give the global well-posedness in $L^2(\R^d)$. For the global well-posedness with large $L^2$-norm initial data, we refer the reader to \cite{PausaderShao10} where the authors established the global well-posedness and scattering for (NL4S) in $L^2(\R^d), d\geq 5$. \newline
\indent The main goal of this paper is to prove the global well-posedness for (NL4S) in a low regularity space $H^\gamma(\R^d), d\geq 5$ with $\gamma <2$. Since we are working with low regularity data, the conservation of energy does not hold. In order to overcome this problem, we make use of the $I$-method and the interaction Morawetz inequality. Due to the high-order term $\Delta^2 u$, we requires the nonlinearity to have at least two orders  of derivatives in order to successfully establish the almost conservation law. We thus restrict ourself in spatial space of dimensions $d=5, 6, 7$. \newline
\indent Let us recall some known results about the global existence below the energy space for the nonlinear fourth-order Schr\"odinger equation. To our knowledge, the first result to address this problem belongs to Guo in \cite{Guo}, where the author considered a more general fourth-order Schr\"odinger equation, namely
\[
i\partial_t u + \lambda \Delta u + \mu \Delta^2 u + \nu |u|^{2m}u =0,
\]
and established the global existence in $H^\gamma(\R^d)$ for $1+\frac{md-9+\sqrt{(4m-md+7)^2+16}}{4m}<\gamma<2$ where $m$ is an integer satisfying $4<md<4m+2$. The proof is based on the $I$-method which is a modification of the one invented by $I$-Team \cite{I-teamalmost} in the context of nonlinear Schr\"odinger equation. Later, Miao-Wu-Zhang studied the defocusing cubic fourth-order Schr\"odinger equation, namely
\[
i\partial_t u +\Delta^2 u  +|u|^2 u=0,
\]
and proved the global well-posedness and scattering in $H^\gamma(\R^d)$ with $\gamma(d)<\gamma<2$ where $\gamma(5)=\frac{16}{11}, \gamma(6)=\frac{16}{9}$ and $\gamma(7)=\frac{45}{23}$. The proof relies on the combination of $I$-method and a new interaction Morawetz inequality. Recently, the author in \cite{Dinhglobal} showed that the defocusing cubic fourth-order Schr\"odinger equation is globally well-posed in $H^\gamma(\R^4)$ with $\frac{60}{53}<\gamma<2$. The analysis is carried out in Bourgain spaces $X^{\gamma,b}$ which is similar to those in \cite{I-teamalmost}. Note that in the above considerations, the nonlinearity is algebraic. This allows to write explicitly the commutator between the $I$-operator and the nonlinearity by means of the Fourier transform, and then control it by multi-linear analysis. When one considers the mass-critical nonlinear fourth-order Schr\"odinger equation in dimensions $d\geq 5$, this method does not work. We thus rely purely on Strichartz and interaction Morawetz estimates. The main result of this paper is the following:
\begin{theorem} \label{theorem global existence}
Let $d= 5, 6, 7$. The initial value problem \emph{(NL4S)} is globally well-posed in $H^\gamma(\R^d)$, for any $\gamma(d)<\gamma<2$, where $\gamma(5)=\frac{8}{5}, \gamma(6)=\frac{5}{3}$ and $\gamma(7)=\frac{13}{7}$.
\end{theorem}  
The proof of the above theorem is based on the combination of the $I$-method and the interaction Morawetz inequality which is similar to those given in \cite{SilvaPavlovicStaffilaniTzirakis}. The $I$-method was first introduced by $I$-Team in \cite{I-teamalmost} in order to treat the nonlinear Schr\"odinger equation at low regularity. The idea is to replace the non-conserved energy $E(u)$ when $\gamma<2$ by an ``almost conserved'' variance $E(Iu)$ with $I$ a smoothing operator which is the identity at low frequency, and behaves like a fractional integral operator of order $2-\gamma$ at high frequency. Since $Iu$ is not a solution of (NL4S), we may expect an energy increment. The key is to show that the modified energy $E(Iu)$ is an ``almost conserved'' quantity in the sense that the time derivative of $E(Iu)$ decays with respect to a large parameter $N$ (see Section $\ref{section preliminaries}$ for the definition of $I$ and $N$). To do so, we need delicate estimates on the commutator between the $I$-operator and the nonlinearity. Note that in our setting, the nonlinearity is not algebraic. Thus we can not apply the Fourier transform technique. Fortunately, thanks to a special Strichartz estimate $(\ref{strichartz estimate biharmonic 4order})$, we are able to apply the technique given in \cite{VisanZhang} to control the commutator. The interaction Morawetz inequality for the nonlinear fourth-order Schr\"odinger equation was first introduced in \cite{Pausadercubic} for $d\geq 7$, and was extended for $d\geq 5$ in \cite{MiaoWuZhang}. With this estimate, the interpolation argument and Sobolev embedding give for any compact interval $J$,
\begin{align}
\|u\|_{M(J)}:=\|u\|_{L^{\frac{8(d-3)}{d}}_t L^{\frac{2(d-3)}{d-4}}_x} \lesssim |J|^{\frac{d-4}{8(d-3)}} \|u_0\|^{\frac{1}{d-3}}_{L^2_x} \|u\|^{\frac{d-4}{d-3}}_{L^\infty_t\dot{H}^{\frac{1}{2}}_x}. \label{morawetz norm}
\end{align}
As a byproduct of the Strichartz estimates and $I$-method, we show the almost conservation law for the modified energy of (NL4S), that is if $u\in C^\infty_0(\R^d)$ is a solution to (NL4S) on a time interval $J=[0,T]$, and satisfies $\|Iu_0\|_{H^2_x}\leq 1$ and if $u$ satisfies in addition the a priori bound $\|u\|_{M(J)} \leq \mu$ for some small constant $\mu>0$, then 
\[
\sup_{t\in[0,T]} |E(Iu(t))-E(Iu_0)| \lesssim N^{-(2-\gamma+\delta)}.
\]
for $\max\left\{3-\frac{8}{d}, \frac{8}{d} \right\}<\gamma<2$ and $0<\delta<\gamma+\frac{8}{d}-3$. \newline 
\indent We now briefly outline the idea of the proof. Let $u$ be a global in time solution to (NL4S). 
Observe that for any $\lambda>0$,
\begin{align}
u_\lambda(t,x):= \lambda^{-\frac{d}{2}} u(\lambda^{-4}t, \lambda^{-1}x) \label{scaling}
\end{align}
is also a solution to (NL4S). By choosing 
\begin{align}
\lambda \sim N^{\frac{2-\gamma}{\gamma}}, \label{lambda introduction}
\end{align}
and using some harmonic analysis, we can make $E(Iu_\lambda(0)) \leq \frac{1}{4}$ by taking $\lambda$ sufficiently large depending on $\|u_0\|_{H^\gamma_x}$ and $N$. Fix an arbitrary large time $T$. The main goal is to show 
\begin{align}
E(Iu_\lambda(\lambda^4T)) \leq 1. \label{almost conservation modified energy}
\end{align} 
With this bound, we can easily obtain the growth of $\|u(T)\|_{H^\gamma_x}$, and the global well-posedness in $H^\gamma(\R^d)$ follows immediately. In order to get $(\ref{almost conservation modified energy})$, we claim that
\[
\|u_\lambda\|_{M([0,t])} \leq K t^{\frac{d-4}{8(d-3)}}, \quad \forall t \in [0,\lambda^4T],
\]
for some constant $K$. If it is not so, then there exists $T_0\in[0,\lambda^4T]$ such that
\begin{align}
\|u_\lambda\|_{M([0,T_0])} &> K T_0^{\frac{d-4}{8(d-3)}}, \label{assumption 1}\\
\|u_\lambda\|_{M([0,T_0])} & \leq 2K T_0^{\frac{d-4}{8(d-3)}}. \label{assumption 2}
\end{align}
Using $(\ref{assumption 2})$, we can split $[0,T_0]$ into $L$ subintervals $J_k, k=1,...,L$ so that
\[
\|u_\lambda\|_{M(J_k)} \leq \mu.
\]
The number $L$ must satisfy
\begin{align}
L \sim T_0^{\frac{d-4}{d}}. \label{L introduction}
\end{align}
Thus we can apply the almost conservation law to get
\[
\sup_{[0,T_0]} E(Iu_\lambda(t)) \leq E(Iu_\lambda(0)) + N^{-(2-\gamma+\delta)} L.
\]
Since $E(Iu_\lambda(0)) \leq \frac{1}{4}$, in order to have $E(Iu_\lambda(t)) \leq 1$ for all $t\in[0,T_0]$, we need
\begin{align}
N^{-(2-\gamma+\delta)} L \ll \frac{1}{4}. \label{smallness introduction}
\end{align}
Combining $(\ref{lambda introduction}), (\ref{L introduction})$ and $(\ref{smallness introduction})$, we obtain the condition on $\gamma$. Next, using $(\ref{morawetz norm})$ together with some harmonic analysis, we estimate
\[
\|u_\lambda\|_{M([0,T_0])} \lesssim T_0^{\frac{d-4}{8(d-3)}} \|u_0\|_{L^2_x}^{\frac{1}{d-3}} \sup_{[0,T_0]} \Big(\|u_0\|^{\frac{3}{4}} \|Iu_\lambda(t)\|_{\dot{H}^2_x}^{\frac{1}{4}} + N^{-\frac{3}{4}} \|Iu_\lambda(t)\|_{\dot{H}^2_x} \Big)^{\frac{d-4}{d-3}}.
\]
Since $\|I u_\lambda(t)\|_{\dot{H}^2_x}\lesssim E(Iu_\lambda(t)) \leq 1$ for all $t\in [0,T_0]$, we get
\[
\|u_\lambda\|_{M([0,T_0])} \leq C T_0^{\frac{d-4}{8(d-3)}},
\]
for some constant $C>0$. This leads to a contradiction to $(\ref{assumption 1})$ for an appropriate choice of $K$. Thus we have the claim and also 
\[
E(Iu_\lambda(t)) \leq 1, \quad \forall t\in [0,\lambda^4T].
\]
For more details, we refer the reader to Section $\ref{section global well-posedness}$. \newline
\indent This paper is organized as follows. In Section $\ref{section preliminaries}$, we introduce some notations and recall some results related to our problem. In Section $\ref{section almost conservation law}$, we show the almost conservation law for the modified energy. Finally, the proof of our main result is given in Section $\ref{section global well-posedness}$.
\section{Preliminaries} \label{section preliminaries}
\setcounter{equation}{0}
In the sequel, the notation $A \lesssim B$ denotes an estimate of the form $A\leq CB$ for some constant $C>0$. The notation $A \sim B$ means that $A \lesssim B$ and $B \lesssim A$. We write $A \ll B$ if $A \leq cB$ for some small constant $c>0$. We also use $\scal{a}:=1+|a|$.
\subsection{Nonlinearity}
Let $F(z):= |z|^{\frac{8}{d}} z, d=5, 6, 7$  be the function that defines the nonlinearity in (NL4S). The derivative $F'(z)$ is defined as a real-linear operator acting on $w \in \C$ by
\[
F'(z)\cdot w:= w \partial_z F(z) + \overline{w} \partial_{\overline{z}} F(z), 
\]
where
\[
\partial_z F(z)=\frac{2d+8}{2d} |z|^{\frac{8}{d}}, \quad \partial_{\overline{z}} F(z) = \frac{4}{d}|z|^{\frac{8}{d}} \frac{z}{\overline{z}}.
\]
We shall identify $F'(z)$ with the pair $(\partial_z F(z), \partial_{\overline{z}} F(z))$, and define its norm by
\[
|F'(z)| := |\partial_zF(z)| + |\partial_{\overline{z}} F(z)|.
\]
It is clear that $|F'(z)| = O(|z|^{\frac{8}{d}})$. 
We also have the following chain rule
\[
\partial_k F(u) = F'(u) \partial_k u, 
\] 
for $k \in \{1,\cdots, d\}$. In particular, we have
\[
\nabla F(u)= F'(u) \nabla u. 
\]
\indent We next recall the fractional chain rule to estimate the nonlinearity.
\begin{lem}\label{lem fractional chain}
Suppose that $G\in C^1(\C, \C)$, and $\alpha \in (0,1)$. Then for $1 <q \leq q_2 <\infty$ and $1<q_1 \leq \infty$ satisfying $\frac{1}{q}=\frac{1}{q_1}+\frac{1}{q_2}$, 
\[
\||\nabla|^\alpha G(u) \|_{L^q_x} \lesssim \|G'(u) \|_{L^{q_1}_x} \||\nabla|^\alpha u \|_{L^{q_2}_x}.
\]
\end{lem}
We refer the reader to \cite[Proposition 3.1]{ChristWeinstein} for the proof of the above estimate when $1<q_1<\infty$, and to \cite[Theorem A.6]{KenigPonceVega} for the proof when $q_1=\infty$. \newline
\indent When $G$ is no longer $C^1$, but H\"older continuous, we have the following fractional chain rule.
\begin{lem}\label{lem fractional chain rule holder}
Suppose that $G\in C^{0,\beta}(\C, \C), \beta \in (0,1)$. Then for every $0<\alpha <\beta, 1<q<\infty$, and $\frac{\alpha}{\beta}<\rho<1$,
\[
\||\nabla|^\alpha G(u)\|_{L^q_x} \lesssim \| |u|^{\beta-\frac{\alpha}{\rho}} \|_{L^{q_1}_x} \||\nabla|^\rho u\|^{\frac{\alpha}{\rho}}_{L^{\frac{\alpha}{\rho}q_2}_x},
\] 
provided $\frac{1}{q}=\frac{1}{q_1}+\frac{1}{q_2}$ and $\left(1-\frac{\alpha}{\beta \rho}\right)q_1>1$.
\end{lem}
The reader can find the proof of this result in \cite[Proposition A.1]{Visanthesis}.
\subsection{Strichartz estimates} \label{subsection fractional derivative}
Let $I \subset \R$ and $p, q \in [1,\infty]$. We define the mixed norm
\[
\|u\|_{L^p_t(I, L^q_x)} := \Big( \int_I \Big( \int_{\R^d} |u(t,x)|^q dx \Big)^{\frac{1}{q}} \Big)^{\frac{1}{p}}
\] 
with a usual modification when either $p$ or $q$ are infinity. When there is no risk of confusion, we may write $L^p_t L^q_x$ instead of $L^p_t(I,L^q_x)$. We also use $L^p_{t,x}$ when $p=q$.
\begin{defi}
A pair $(p,q)$ is said to be \textbf{Schr\"odinger admissible}, for short $(p,q) \in S$, if 
\[
(p,q) \in [2,\infty]^2, \quad (p,q,d) \ne (2,\infty,2), \quad \frac{2}{p}+\frac{d}{q} \leq \frac{d}{2}.
\]
\end{defi}
We also denote for $(p,q)\in [1,\infty]^2$,
\begin{align}
\gamma_{p,q}=\frac{d}{2}-\frac{d}{q}-\frac{4}{p}. \label{define gamma pq}
\end{align}
\begin{defi}
A pair $(p,q)$ is called \textbf{biharmonic admissible}, for short $(p,q)\in B$, if 
\[
(p,q) \in S, \quad \gamma_{p,q}=0.
\]
\end{defi}
\begin{prop}[Strichartz estimate for fourth-order Schr\"odinger equation \cite{Dinhfract}] \label{prop strichartz}
Let $\gamma \in \R$ and $u$ be a (weak) solution to the linear fourth-order Schr\"odinger equation namely
\[
u(t)= e^{it\Delta^2}u_0 + \int_0^t e^{i(t-s)\Delta^2} F(s) ds,
\]
for some data $u_0, F$. Then for all $(p,q)$ and $(a,b)$ Schr\"odinger admissible with $q<\infty$ and $b<\infty$,
\begin{align}
\||\nabla|^\gamma u\|_{L^p_t(\R, L^q_x)} \lesssim \||\nabla|^{\gamma+\gamma_{p,q}} u_0\|_{L^2_x} + \||\nabla|^{\gamma+\gamma_{p,q}-\gamma_{a',b'} -4} F\|_{L^{a'}_t(\R, L^{b'}_x)}. \label{strichartz estimate}
\end{align}
Here $(a,a')$ and $(b,b')$ are conjugate pairs, and $\gamma_{p,q}, \gamma_{a',b'}$ are defined as in $(\ref{define gamma pq})$.
\end{prop}
Note that the estimate $(\ref{strichartz estimate})$ is exactly the one given in \cite{MiaoZhang}, \cite{Pausader} or \cite{Pausadercubic} where the author considered $(p,q)$ and $(a,b)$ are either sharp Schr\"odinger admissible, i.e.
\[
p, q \in [2,\infty]^2, \quad (p,q,d) \ne (2,\infty,2), \quad \frac{2}{p}+\frac{d}{q}=\frac{d}{2},
\]
or biharmonic admissible.
We refer the reader to \cite[Proposition 2.1]{Dinhfract} for the proof of Proposition $\ref{prop strichartz}$. The proof is based on the scaling technique instead of using a dedicate dispersive estimate of \cite{Ben-ArtziKochSaut} for the fundamental solution of the homogeneous fourth-order Schr\"odinger equation. \newline
\indent The following result is a direct consequence of $(\ref{strichartz estimate})$. 
\begin{coro} \label{coro strichartz}
Let $u$ be a (weak) solution to the linear fourth-order Schr\"odinger equation for some data $u_0, F$. Then for all $(p,q)$ and $(a,b)$ biharmonic admissible satisfying $q<\infty$ and $b<\infty$,
\begin{align}
\|u\|_{L^p_t(\R,L^q_x)} \lesssim \|u_0\|_{L^2_x} + \|F\|_{L^{a'}_t(\R,L^{b'}_x)}, \label{strichartz estimate biharmonic}
\end{align}
and
\begin{align}
\|\Delta u\|_{L^p_t(\R,L^q_x)} \lesssim \|\Delta u_0\|_{L^2_x} + \|\nabla F\|_{L^2_t(\R,L^{\frac{2d}{d+2}}_x)}. \label{strichartz estimate biharmonic 4order}
\end{align}
\end{coro}
\subsection{Littlewood-Paley decomposition}
Let $\varphi$ be a radial smooth bump function supported in the ball $|\xi|\leq 2$ and  equal to 1 on the ball $|\xi|\leq 1$.  For $M=2^k, k \in \Z$, we define the Littlewood-Paley operators:
\begin{align*}
\widehat{P_{\leq M} f}(\xi) &:= \varphi(M^{-1}\xi) \hat{f}(\xi), \\
\widehat{P_{>M} f}(\xi) &:= (1-\varphi(M^{-1}\xi)) \hat{f}(\xi), \\
\widehat{P_M f}(\xi) &:= (\varphi(M^{-1} \xi) - \varphi(2M^{-1}\xi)) \hat{f}(\xi),
\end{align*}
where $\hat{\cdot}$ is the spatial Fourier transform. Similarly, we can define
\[
P_{<M} := P_{\leq M}-P_M, \quad P_{\geq M} := P_{>M}+ P_M,
\] 
and for $M_1 \leq M_2$,
\[
P_{M_1 < \cdot \leq M_2}:= P_{\leq M_2} - P_{\leq M_1} = \sum_{M_1 < M \leq M_2} P_M.
\]
We recall the following standard Bernstein inequalities (see e.g. \cite[Chapter 2]{BCDfourier} or \cite[Appendix]{Tao}):
\begin{lem}[Bernstein inequalities] \label{lem bernstein}
Let $\gamma\geq 0$ and $1 \leq p \leq q \leq \infty$. We have
\begin{align*}
\|P_{\geq M} f\|_{L^p_x} &\lesssim M^{-\gamma} \||\nabla|^\gamma P_{\geq M} f\|_{L^p_x}, \\
\|P_{\leq M} |\nabla|^\gamma f\|_{L^p_x} &\lesssim M^\gamma \|P_{\leq M} f\|_{L^p_x}, \\
\|P_M |\nabla|^{\pm \gamma} f\|_{L^p_x} & \sim M^{\pm \gamma} \|P_M f\|_{L^p_x}, \\
\|P_{\leq M} f\|_{L^q_x} &\lesssim M^{\frac{d}{p}-\frac{d}{q}} \|P_{\leq M} f\|_{L^p_x}, \\
\|P_M f\|_{L^q_x} &\lesssim M^{\frac{d}{p}-\frac{d}{q}} \|P_M f\|_{L^p_x}.
\end{align*}
\end{lem}
\subsection{$I$-operator}
Let $0\leq \gamma <2$ and $N\gg 1$. We define the Fourier multiplier $I_N$ by
\[
\widehat{I_N f}(\xi):= m_N(\xi) \hat{f}(\xi),
\]
where $m_N$ is a smooth, radially symmetric, non-increasing function such that 
\begin{align*}
m_N(\xi) := \left\{ 
\begin{array}{cl}
1 &\text{if } |\xi|\leq N, \\
(N^{-1}|\xi|)^{\gamma-2} & \text{if } |\xi| \geq 2N.
\end{array}
\right. 
\end{align*}
We shall drop the $N$ from the notation and write $I$ and $m$ instead of $I_N$ and $m_N$. We collect some basic properties of the $I$-operator in the following lemma.
\begin{lem} \label{lem properties I operator}
Let $0\leq \sigma \leq \gamma<2$ and $1<q<\infty$. Then
\begin{align}
\|I f\|_{L^q_x} &\lesssim \|f\|_{L^q_x}, \label{property 1} \\
\| |\nabla|^\sigma P_{>N} f\|_{L^q_x} &\lesssim N^{\sigma-2} \|\Delta I f\|_{L^q_x}, \label{property 2} \\
\|\scal{\nabla}^\sigma f\|_{L^q_x} &\lesssim \|\scal{\Delta} I f\|_{L^q_x}, \label{property 3} \\
\|f\|_{H^\gamma_x} \lesssim \|If\|_{H^2_x} &\lesssim N^{2-\gamma} \|f\|_{H^\gamma_x}, \label{property 4} \\
\|If\|_{\dot{H}^2_x} &\lesssim N^{2-\gamma} \|f\|_{\dot{H}^\gamma_x}. \label{property 5}
\end{align}
\end{lem}
\begin{proof}
The estimate $(\ref{property 1})$ is a direct consequence of the H\"ormander-Mikhlin multiplier theorem. To prove $(\ref{property 2})$, we write
\[
\||\nabla|^\sigma P_{>N}f\|_{L^q_x} = \| |\nabla|^\sigma P_{>N} (\Delta I)^{-1} \Delta I f\|_{L^q_x}.
\]
The desired estimate $(\ref{property 2})$ follows again from the H\"ormander-Mikhlin multiplier theorem. In order to get $(\ref{property 3})$, we estimate
\[
\|\scal{\nabla}^\sigma f\|_{L^q_x} \leq \|P_{\leq N} \scal{\nabla}^\sigma f\|_{L^q_x}+ \|P_{>N} f\|_{L^q_x}+  \|P_{>N} |\nabla|^\sigma f\|_{L^q_x}.
\]
Thanks to the fact that the $I$-operator is the identity at low frequency $|\xi|\leq N$, the multiplier theorem and $(\ref{property 2})$ imply
\[
\|\scal{\nabla}^\sigma f\|_{L^q_x} \lesssim \|\scal{\Delta} I f\|_{L^q_x} + \|\Delta I f\|_{L^q_x}.
\]
This proves $(\ref{property 3})$. Finally, by the definition of the $I$-operator and $(\ref{property 2})$, we have
\begin{align*}
\|f\|_{H^\gamma_x} &\lesssim \|P_{\leq N} f\|_{H^\gamma_x} + \|P_{>N} f\|_{L^2_x} + \| |\nabla|^\gamma P_{>N}f\|_{L^2_x} \\
&\lesssim \|P_{\leq N} I f\|_{H^\gamma_x} + N^{-2} \|\Delta I f\|_{L^2_x} + N^{\gamma-2} \|\Delta I f\|_{L^2_x} \lesssim \|If\|_{H^2_x}.
\end{align*}
This shows the first inequality in $(\ref{property 4})$. For the second inequality in $(\ref{property 4})$, we estimate
\[
\|I f\|_{H^2_x} \lesssim \|P_{\leq N}\scal{\nabla}^2If\|_{L^2_x} + \|P_{>N} \scal{\nabla}^2If\|_{L^2_x} \lesssim N^{2-\gamma}\|f\|_{H^\gamma_x}.
\]
Here we use the definition of $I$-operator to get
\[
\|P_{\leq N} I \scal{\nabla}^{2-\gamma}\|_{L^2_x \rightarrow L^2_x}, \quad \|P_{>N} I \scal{\nabla}^{2-\gamma}\|_{L^2_x \rightarrow L^2_x} \lesssim N^{2-\gamma}.
\]
The estimate $(\ref{property 5})$ is proved as for the second estimate in $(\ref{property 4})$. The proof is complete.
\end{proof}
When the nonlinearity $F(u)$ is algebraic, one can use the Fourier transform to write the commutator like $F(Iu)-IF(u)$ as a product of Fourier transforms of $u$ and $Iu$, and then measure the frequency interactions. However, in our setting, the nonlinearity is no longer algebraic, we thus need the following rougher estimate which is a modified version of the Schr\"odinger context (see \cite{VisanZhang}).
\begin{lem} \label{lem rougher estimate}
Let $1<\gamma<2, 0<\delta <\gamma-1$ and $1<q, q_1, q_2 <\infty$ be such that $\frac{1}{q}=\frac{1}{q_1}+\frac{1}{q_2}$. Then
\begin{align}
\|I(fg)-(If)g\|_{L^q_x} \lesssim N^{-(2-\gamma+\delta)} \|If\|_{L^{q_1}_x} \|\scal{\nabla}^{2-\gamma+\delta} g\|_{L^{q_2}_x}.
\end{align}
\end{lem}  
The proof is a slight modification of the one given in Lemma 2.5 of \cite{VisanZhang}. We thus only give a sketch of the proof.\newline
\textit{Sketch of the proof.} 
By the Littlewood-Paley decomposition, we write
\begin{align*}
I(fg)-(If)g &= I(f P_{\leq 1}g) - (If) P_{\leq 1}g + \sum_{M>1} [I(P_{\lesssim M}f P_Mg) - (IP_{\lesssim M}f) P_Mg] \\
&\mathrel{\phantom{=}} + \sum_{M>1} [I(P_{\gg M}fP_Mg)- (IP_{\gg M} f)P_Mg] \\
&= I(P_{\gtrsim N}f P_{\leq 1}g) - (I P_{\gtrsim N}f) P_{\leq 1}g + \sum_{M\gtrsim N} [I(P_{\lesssim M} f P_Mg) - (IP_{\lesssim M}f) P_Mg] \\
&\mathrel{\phantom{=}} + \sum_{M>1} [I(P_{\gg M}fP_Mg)- (IP_{\gg M}f)P_Mg] \\
&= \text{Term}_1 + \text{Term}_2 + \text{Term}_3.
\end{align*}
Here we use the definition of the $I$-operator to get
\[
I(P_{\ll N}f P_{\leq 1}g) = (I P_{\ll N}f) P_{\leq 1}g, \quad I(P_{\lesssim M}f P_Mg) = (IP_{\lesssim M}f) P_M g,
\]
for all $M \ll N$. \newline
\indent For the second term, using Lemma $\ref{lem bernstein}$ and Lemma $\ref{lem properties I operator}$, we estimate
\begin{align*}
\|I(P_{\lesssim M}f P_Mg) - (IP_{\lesssim M}f) P_Mg \|_{L^q_x} &\lesssim \|P_{\lesssim M}f\|_{L^{q_1}_x} \|P_Mg\|_{L^{q_2}_x}, \quad M \gtrsim N \\
&\lesssim \Big(\frac{M}{N}\Big)^{2-\gamma} \|If\|_{L^{q_1}_x} \|P_Mg\|_{L^{q_2}_x} \\
&\lesssim M^{-\delta} N^{-(2-\gamma)}\|If\|_{L^{q_1}_x} \||\nabla|^{2-\gamma+\delta} g\|_{L^{q_2}_x}.
\end{align*}
Summing over all $N\lesssim M \in 2^\Z$, we get
\[
\|\text{Term}_2\|_{L^q_x} \lesssim N^{-(2-\gamma+\delta)} \|If\|_{L^{q_1}_x} \||\nabla|^{2-\gamma+\delta} g\|_{L^{q_2}_x}.
\]
For the third term, we write
\begin{align*}
I(P_{\gg M} f P_M g)- (IP_{\gg M}f) P_M g &= \sum_{1\ll k \in \N} [ I(P_{2^kM}f P_Mg) - (I P_{2^kM} f) P_Mg] \\
&= \sum_{1\ll k \in \N \atop N\lesssim 2^k M} I(P_{2^k M} f P_M g) - (IP_{2^k M}f) P_Mg].
\end{align*}
We note that
\[
[I(P_{2^kM}f P_Mg)-(IP_{2^kM}f) P_Mg]\widehat{\ }(\xi) = \int_{\xi=\xi_1 +\xi_2} (m_N(\xi_1+\xi_2) -m_N(\xi_1)) \widehat{P_{2^kM} f}(\xi_1) \widehat{P_Mg}(\xi_2).
\]
For $|\xi_1|\sim 2^kM \gtrsim N$ and $|\xi_2|\sim M$, the mean value theorem implies
\[
|m_N(\xi_1+\xi_2)-m_N(\xi_2)| \lesssim |\nabla m_N(\xi_1)| |\xi_2| \lesssim 2^{-k} \Big( \frac{2^kM}{N}\Big)^{\gamma-2}.
\]
The Coifman-Meyer multiplier theorem (see e.g. \cite{CoifmanMeyer, CoifmanMeyer91}) then yields
\begin{align*}
\|I(P_{2^kM}f P_Mg)-(IP_{2^kM}f) P_M g\|_{L^q_x} &\lesssim 2^{-k} \Big( \frac{2^kM}{N}\Big)^{\gamma-2} \|P_{2^kM}f\|_{L^{q_1}_x} \|P_Mg\|_{L^{q_2}_x} \\
&\lesssim 2^{-k} M^{-(2-\gamma+\delta)} \|If\|_{L^{q_1}_x} \||\nabla|^{2-\gamma+\delta} g\|_{L^{q_2}_x}.
\end{align*}
By rewrite $2^{-k} M^{-(2-\gamma+\delta)}= 2^{-k(\gamma-1-\delta)} (2^kM)^{-(2-\gamma+\delta)}$, we sum over all $k\gg 1$ with $\gamma-1>\delta$ and $N \lesssim 2^kM$ to get
\[
\|\text{Term}_3\|_{L^q_x} \lesssim N^{-(2-\gamma+\delta)} \|If\|_{L^{q_1}_x} \||\nabla|^{2-\gamma+\delta} g\|_{L^{q_2}_x}.
\]
Finally, we consider the first term. It is proved by the same argument as for the third term. We estimate
\begin{align*}
\|\text{Term}_1\|_{L^q_x} &\lesssim \sum_{k\in \N, 2^k \gtrsim N} \|I(P_{2^k}f P_{\leq 1}g) - (IP_{2^k}f)P_{\leq 1}g\|_{L^q_x} \\
&\lesssim \sum_{k\in \N, 2^k \gtrsim N} 2^{-k} \|If\|_{L^{q_1}_x} \|g\|_{L^{q_2}_x} \\
&\lesssim N^{-1} \|If\|_{L^{q_1}_x} \|g\|_{L^{q_2}_x}.
\end{align*}
Note that the condition $\gamma-1>\delta$ ensures that $N^{-1} \lesssim N^{-(2-\gamma+\delta)}$. This completes the proof. 
\defendproof \newline
\indent As a direct consequence of Lemma $\ref{lem rougher estimate}$ with the fact that
\[
\nabla F(u) = \nabla u F'(u),
\]
we have the following corollary. Note that the $I$-operator commutes with $\nabla$.
\begin{coro} \label{coro rougher estimate}
Let $1<\gamma<2, 0<\delta<\gamma-1$ and $1<q, q_1, q_2<\infty$ be such that $\frac{1}{q}=\frac{1}{q_1}+\frac{1}{q_2}$. Then 
\begin{align}
\|\nabla I F(u)-(I\nabla u)F'(u)\|_{L^q_x} \lesssim N^{-(2-\gamma+\delta)} \|\nabla I u\|_{L^{q_1}_x} \|\scal{\nabla}^{2-\gamma+\delta} F'(u)\|_{L^{q_2}_x}. \label{rougher estimate 1} 
\end{align}
\end{coro}
\subsection{Interaction Morawetz inequality}
We end this section by recalling the interaction Morawetz inequality for the nonlinear fourth-order Schr\"odinger equation. This estimate was first established by Pausader in \cite{Pausadercubic} for $d\geq 7$. Later, Miao-Wu-Zhang in \cite{MiaoWuZhang} extended this interaction Morawetz estimate to $d\geq 5$. 
\begin{prop}[Interaction Morawetz inequality \cite{Pausadercubic}, \cite{MiaoWuZhang}]
Let $d\geq 5$, $J$ be a compact time interval and $u$ a solution to \emph{(NL4S)} on the spacetime slab $J\times \R^d$. Then we have the following a priori estimate:
\begin{align}
\||\nabla|^{-\frac{d-5}{4}} u\|_{L^4_t(J, L^4_x)} \lesssim \|u_0\|^{\frac{1}{2}}_{L^2_x} \|u\|^{\frac{1}{2}}_{L^\infty_t(J, \dot{H}^{\frac{1}{2}}_x)}. \label{interaction morawetz}
\end{align} 
\end{prop}
We have from $(\ref{interaction morawetz})$ and the Sobolev embedding that
\begin{align}
\|u\|_{L^4_t(J,L^{\frac{4d}{2d-5}}_x)} \lesssim \|u_0\|^{\frac{1}{2}}_{L^2_x} \|u\|^{\frac{1}{2}}_{L^\infty_t(J, \dot{H}^{\frac{1}{2}}_x)}. \label{interaction morawetz sobolev}
\end{align}
By interpolating between $(\ref{interaction morawetz sobolev})$ and the trivial estimate 
\[
\|u\|_{L^\infty_t(J, L^{\frac{2d}{d-1}}_x)} \lesssim \|u\|_{L^\infty_t(J, \dot{H}^{\frac{1}{2}}_x)},
\]
we obtain
\begin{align*}
\|u\|_{L^{2(d-3)}_t(J, L^{\frac{2(d-3)}{d-4}}_x)} \lesssim \Big(\|u_0\|_{L^2_x} \|u\|_{L^\infty_t(J,\dot{H}^{\frac{1}{2}}_x)}\Big)^{\frac{1}{d-3}} \|u\|_{L^\infty_t(J,\dot{H}^{\frac{1}{2}}_x)}^{\frac{d-5}{d-3}}= \|u_0\|_{L^2_x}^{\frac{1}{d-3}} \|u\|^{\frac{d-4}{d-3}}_{L^\infty_t(J,\dot{H}^{\frac{1}{2}}_x)}. 
\end{align*}
Using Sobolev embedding in time, we get
\begin{align}
\|u\|_{M(J)}:=\|u\|_{L^{\frac{8(d-3)}{d}}_t (J, L^{\frac{2(d-3)}{d-4}}_x)} \lesssim |J|^{\frac{d-4}{8(d-3)}} \|u_0\|^{\frac{1}{d-3}}_{L^2_x} \|u\|^{\frac{d-4}{d-3}}_{L^\infty_t(J,\dot{H}^{\frac{1}{2}}_x)}. \label{interaction morawetz interpolation}
\end{align}
Here $\Big(\frac{8(d-3)}{d}, \frac{2(d-3)}{d-4}\Big)$ is a biharmonic admissible pair. 
\section{Almost conservation law} \label{section almost conservation law}
\setcounter{equation}{0}
For any spacetime slab $J\times \R^d$, we define
\begin{align}
Z_I(J):= \sup_{(p,q)\in B, q<\infty} \|\scal{\Delta} I u\|_{L^p_t(J, L^q_x)}. \nonumber
\end{align}
Let us start with the following commutator estimates.
\begin{lem} \label{lem commutator estimate}
Let $5\leq d\leq 7, 1<\gamma<2$ and $0<\delta <\gamma-1$. Then
\begin{align}
\|\nabla I F(u)-(I\nabla u) F'(u)\|_{L^2_t(J, L^{\frac{2d}{d+2}}_x)} &\lesssim N^{-(2-\gamma+\delta)} (Z_I(J))^{1+\frac{8}{d}}, \label{commutator estimate 1} \\
\|\nabla I F(u)\|_{L^2_t(J, L^{\frac{2d}{d+2}}_x)} &\lesssim \|u\|^{\frac{8}{d}}_{M(J)} Z_I(J) + N^{-(2-\gamma+\delta)}(Z_I(J))^{1+\frac{8}{d}}, \label{commutator estimate 2}
\end{align}
where $\|u\|_{M(J)}$ is given in $(\ref{interaction morawetz interpolation})$. In particular, 
\begin{align}
\|\nabla I F(u)\|_{L^2_t(J, L^{\frac{2d}{d+2}}_x)} \lesssim (Z_I(J))^{1+\frac{8}{d}}. \label{commutator estimate 3}
\end{align}
\end{lem}
\begin{proof}
For simplifying the notation, we shall drop the dependence on the time interval $J$. 
We apply $(\ref{rougher estimate 1})$ with $q=\frac{2d}{d+2}, q_1=\frac{2d(d-3)}{d^2-9d+22}$ and $q_2=\frac{d(d-3)}{2(2d-7)}$ to get
\[
\|\nabla I F(u)-(I\nabla u) F'(u)\|_{L^{\frac{2d}{d+2}}_x} \lesssim N^{-(2-\gamma+\delta)} \|\nabla I u\|_{L^{\frac{2d(d-3)}{d^2-9d+22}}_x} \|\scal{\nabla}^{2-\gamma+\delta}F'(u)\|_{L^{\frac{d(d-3)}{2(2d-7)}}_x}.
\]
We then apply H\"older's inequality to have
\[
\|\nabla I F(u)-(I\nabla u) F'(u)\|_{L^2_tL^{\frac{2d}{d+2}}_x} \lesssim N^{-\alpha} \|\nabla I u\|_{L^{\frac{2(d-3)}{d-4}}_tL^{\frac{2d(d-3)}{d^2-9d+22}}_x} \|\scal{\nabla}^\alpha F'(u)\|_{L^{2(d-3)}_tL^{\frac{d(d-3)}{2(2d-7)}}_x},
\] 
where $\alpha=2-\gamma+\delta \in (0,1)$ by our assumptions. 
For the first factor in the right hand side, we use the Sobolev embedding to obtain
\begin{align}
\|\nabla I u\|_{L^{\frac{2(d-3)}{d-4}}_tL^{\frac{2d(d-3)}{d^2-9d+22}}_x} \lesssim \|\Delta I u\|_{L^{\frac{2(d-3)}{d-4}}_t L^{\frac{2d(d-3)}{d^2-7d+16}}_x} \lesssim Z_I, \label{estimate 1}
\end{align}
where $\Big(\frac{2(d-3)}{d-4}, \frac{2d(d-3)}{d^2-7d+16}\Big)$ is a biharmonic admissible pair. For the second factor, we estimate
\begin{align}
\|\scal{\nabla}^\alpha F'(u)\|_{L^{2(d-3)}_tL^{\frac{d(d-3)}{2(2d-7)}}_x} \lesssim \|F'(u)\|_{L^{2(d-3)}_tL^{\frac{d(d-3)}{2(2d-7)}}_x} + \||\nabla|^\alpha F'(u)\|_{L^{2(d-3)}_tL^{\frac{d(d-3)}{2(2d-7)}}_x}. \label{estimate commutator}
\end{align}
Since $F'(u)=O(|u|^{\frac{8}{d}})$, we use $(\ref{property 3})$ to have
\begin{align}
\|F'(u)\|_{L^{2(d-3)}_tL^{\frac{d(d-3)}{2(2d-7)}}_x} \lesssim \|u\|^{\frac{8}{d}}_{L^{\frac{16(d-3)}{d}}_tL^{\frac{4(d-3)}{2d-7}}_x} \lesssim Z_I^{\frac{8}{d}}, \label{estimate 2}
\end{align}
where $\Big(\frac{16(d-3)}{d}, \frac{4(d-3)}{2d-7}\Big)$ is biharmonic admissible. In order to treat the second term in $(\ref{estimate commutator})$, 
we apply Lemma $\ref{lem fractional chain}$ with $q=\frac{d(d-3)}{2(2d-7)}, q_1 = \frac{2d(d-3)}{-d^2+11d-26}$ and $q_2=\frac{2d(d-3)}{d^2-3d-2}$ to get
\[
\||\nabla|^\alpha F'(u)\|_{L^{2(d-3)}_tL^{\frac{d(d-3)}{2(2d-7)}}_x} \lesssim \|F''(u)\|_{L^{4(d-3)}_t L^{\frac{2d(d-3)}{-d^2+11d-26}}_x} \||\nabla|^\alpha u\|_{L^{4(d-3)}_t L^{\frac{2d(d-3)}{d^2-3d-2}}_x}.
\]
As $F''(u)=O(|u|^{\frac{8}{d}-1})$, we have
\begin{align}
\|F''(u)\|_{L^{4(d-3)}_t L^{\frac{2d(d-3)}{-d^2+11d-26}}_x} \lesssim \|u\|^{\frac{8}{d}-1}_{L^{\frac{4(8-d)(d-3)}{d}}_t L^{\frac{2(8-d)(d-3)}{-d^2+11d-26}}_x} \lesssim Z_I^{\frac{8}{d}-1}. \label{estimate 3}
\end{align}
Here $\Big( \frac{4(8-d)(d-3)}{d}, \frac{2(8-d)(d-3)}{-d^2+11d-26} \Big)$ is biharmonic admissible. Since $\Big(4(d-3), \frac{2d(d-3)}{d^2-3d-2} \Big)$ is also a biharmonic admissible, we have from $(\ref{property 3})$ that
\begin{align}
\||\nabla|^\alpha u\|_{L^{4(d-3)}_t L^{\frac{2d(d-3)}{d^2-3d-2}}_x} \lesssim Z_I. \label{estimate 4}
\end{align}
Note that $\alpha<1<\gamma$.
Collecting $(\ref{estimate 1}), (\ref{estimate 2}), (\ref{estimate 3})$ and $(\ref{estimate 4})$, we prove $(\ref{commutator estimate 1})$. \newline
\indent We now prove $(\ref{commutator estimate 2})$. We have from $(\ref{commutator estimate 1})$ and the triangle inequality that
\begin{align}
\|\nabla I F(u)\|_{L^2_t L^{\frac{2d}{d+2}}_x} \lesssim \|(\nabla I u) F'(u)\|_{L^2_t L^{\frac{2d}{d+2}}_x} + N^{-(2-\gamma+\delta)} Z_I^{1+\frac{8}{d}}. \label{estimate 6}
\end{align}
The H\"older inequality gives
\[
 \|(\nabla I u) F'(u)\|_{L^2_t L^{\frac{2d}{d+2}}_x} \lesssim \|\nabla I u\|_{L^{\frac{2(d-3)}{d-5}}_tL^{\frac{2d(d-3)}{d^2-9d+26}}_x} \|F'(u)\|_{L^{d-3}_t L^{\frac{d(d-3)}{4(d-4)}}_x}.
\]
We use the Sobolev embedding to estimate
\begin{align}
\|\nabla I u\|_{L^{\frac{2(d-3)}{d-5}}_tL^{\frac{2d(d-3)}{d^2-9d+26}}_x} \lesssim \|\Delta I u\|_{L^{\frac{2(d-3)}{d-5}}_t L^{\frac{2d(d-3)}{d^2-7d+20}}_x} \lesssim Z_I. \label{estimate 7}
\end{align}
Here $\Big(\frac{2(d-3)}{d-5}, \frac{2d(d-3)}{d^2-7d+20} \Big)$ is biharmonic admissible. Since $F'(u)=O(|u|^{\frac{8}{d}})$, we have
\begin{align}
\|F'(u)\|_{L^{d-3}_t L^{\frac{d(d-3)}{4(d-4)}}_x} \lesssim \|u\|^{\frac{8}{d}}_{L^{\frac{8(d-3)}{d}}_t L^{\frac{2(d-3)}{d-4}}_x} = \|u\|^{\frac{8}{d}}_{M}. \label{estimate 8}
\end{align}
Combining $(\ref{estimate 6}), (\ref{estimate 7})$ and $(\ref{estimate 8})$, we obtain $(\ref{commutator estimate 2})$. The estimate $(\ref{commutator estimate 3})$ follows directly from $(\ref{commutator estimate 2})$ and $(\ref{property 3})$. Note that $\Big(\frac{8(d-3)}{d}, \frac{2(d-3)}{d-4}\Big)$ is biharmonic admissible. The proof is complete.
\end{proof}
We are now able to prove the almost conservation law for the modified energy functional $E(Iu)$, where
\[
E(Iu(t))= \frac{1}{2}\|Iu(t)\|^2_{\dot{H}^2_x} + \frac{d}{2d+8}\|Iu(t)\|^{\frac{2d+8}{d}}_{L^{\frac{2d+8}{d}}_x}.
\]
\begin{prop}\label{prop almost conservation law}
Let $5\leq d\leq 7, \max\left\{3-\frac{8}{d},\frac{8}{d}\right\}<\gamma<2$ and $0<\delta<\gamma+\frac{8}{d}-3$. Assume that $u \in L^\infty([0,T],\mathscr{S}(\R^d))$ is a solution to \emph{(NL4S)} on a time interval $J=[0,T]$, and satisfies $\|Iu_0\|_{H^2_x}\leq 1$. Assume in addition that $u$ satisfies the a priori bound
\[
\|u\|_{M(J)} \leq \mu,
\] 
for some small constant $\mu>0$. Then, for $N$ sufficiently large,
\begin{align}
\sup_{t\in [0,T]} |E(Iu(t))-E(Iu_0)| \lesssim N^{-(2-\gamma+\delta)}. \label{almost conservation law}
\end{align}
Here the implicit constant depends only on the size fo $E(Iu_0)$.
\end{prop}
\begin{proof}
We again drop the notation $J$ for simplicity. Our first step is to control the size of $Z_I$. Applying $I$, $\Delta I $ to (NL4S), and then using Strichartz estimates $(\ref{strichartz estimate biharmonic}), (\ref{strichartz estimate biharmonic 4order})$, we have
\begin{align}
Z_I \lesssim \|I u_0\|_{H^2_x} + \|IF(u)\|_{L^2_t L^{\frac{2d}{d+4}}_x} + \|\nabla I F(u)\|_{L^2_t L^{\frac{2d}{d+2}}_x}. \label{local estimate 1}
\end{align}
Using $(\ref{commutator estimate 2})$, we have
\begin{align}
\|\nabla I F(u)\|_{L^2_t L^{\frac{2d}{d+2}}_x} \lesssim \|u\|^{\frac{8}{d}}_M Z_I + N^{-(2-\gamma+\delta)}Z_I^{1+\frac{8}{d}} \lesssim \mu^{\frac{8}{d}} Z_I + N^{-(2-\gamma+\delta)}Z_I^{1+\frac{8}{d}}. \label{local estimate 2}
\end{align}
We next drop the $I$-operator and use H\"older's inequality to estimate
\begin{align}
\|IF(u)\|_{L^2_t L^{\frac{2d}{d+4}}_x} &\lesssim \||u|^{\frac{8}{d}} \|_{L^{d-3}_t L^{\frac{d(d-3)}{4(d-4)}}_x} \|u\|_{L^{\frac{2(d-3)}{d-5}}_t L^{\frac{2d(d-3)}{d^2-7d+20}}_x} \nonumber \\
&\lesssim \|u\|^{\frac{8}{d}}_{L^{\frac{8(d-3)}{d}}_t L^{\frac{2(d-3)}{d-4}}_x} \|u\|_{L^{\frac{2(d-3)}{d-5}}_t L^{\frac{2d(d-3)}{d^2-7d+20}}_x} \nonumber\\
&\lesssim \|u\|^{\frac{8}{d}}_M Z_I \lesssim \mu^{\frac{8}{d}}Z_I. \label{local estimate 3}
\end{align}
The last inequality follows from $(\ref{property 3})$ and the fact $\Big(\frac{2(d-3)}{d-5}, \frac{2d(d-3)}{d^2-7d+20}\Big)$ is biharmonic admissible. Collecting from $(\ref{local estimate 1})$ to $(\ref{local estimate 3})$, we obtain
\[
Z_I \lesssim \|Iu_0\|_{H^2_x} + \mu^{\frac{8}{d}}Z_I + N^{-(2-\gamma+\delta)} Z_I^{1+\frac{8}{d}}.
\]
By taking $\mu$ sufficiently small and $N$ sufficiently large, the continuity argument gives 
\begin{align}
Z_I \lesssim \|Iu_0\|_{H^2_x} \leq 1. \label{control size Z_I}
\end{align}
\indent Next, we have from a direct computation that
\[
\partial_t E(Iu(t)) = \re{ \int \overline{I\partial_t u} (\Delta^2 Iu + F(Iu))  dx}.
\]
By the Fundamental Theorem of Calculus, 
\[
E(Iu(t))-E(Iu_0) = \int_0^t \partial_s E(Iu(s)) ds  = \re{ \int_0^t \int \overline{I \partial_s u} (\Delta^2 Iu + F(Iu)) dx ds}.  
\]
Using $I\partial_t u = i\Delta^2 Iu + i IF(u)$, we see that
\begin{align}
E(Iu(t))-E(Iu_0) &= \re{ \int_0^t \int \overline{I \partial_s u} (F(Iu)-IF(u)) dx ds} \nonumber \\
&= \im{ \int_0^t \int \overline{\Delta^2 Iu + IF(u)} (F(Iu)-IF(u)) dx ds} \nonumber \\
&= \im{ \int_0^t \int \overline{\Delta Iu} \Delta(F(Iu)-IF(u)) dx ds } \nonumber\\
&\mathrel{\phantom{=}} + \im{\int_0^t \int \overline{IF(u)} (F(Iu)-IF(u)) dx ds}. \nonumber 
\end{align}
We next write 
\begin{align*}
\Delta(F(Iu)- IF(u)) &= (\Delta I u) F'(Iu) + |\nabla I u|^2F''(Iu) - I(\Delta F'(u)) - I(|\nabla u|^2 F''(u)) \\
&= (\Delta I u) (F'(Iu)-F'(u)) + |\nabla I u|^2(F''(Iu)-F''(u))  + \nabla Iu \cdot (\nabla I u - \nabla u) F''(u)  \\
& \mathrel{\phantom{=}} + (\Delta I u) F'(u) -I(\Delta u F'(u)) + (I\nabla u)\cdot \nabla u F''(u) - I(\nabla u \cdot \nabla u F''(u)).
\end{align*}
Therefore,
\begin{align}
E(Iu(t))-E(Iu_0) &= \im{ \int_0^t \int \overline{\Delta Iu} \Delta I u (F'(Iu)-F'(u)) dx ds} \label{almost estimate 1} \\
& \mathrel{\phantom{=}} + \im{ \int_0^t \int \overline{\Delta Iu} |\nabla I u|^2(F''(Iu)-F''(u)) dx ds} \label{almost estimate 2} \\
& \mathrel{\phantom{=}} + \im{ \int_0^t \int \overline{\Delta Iu} \nabla Iu \cdot (\nabla I u - \nabla u) F''(u) dx ds} \label{almost estimate 3} \\
& \mathrel{\phantom{=}} + \im{ \int_0^t \int \overline{\Delta Iu} [(\Delta I u) F'(u) -I(\Delta u F'(u))] dx ds} \label{almost estimate 4} \\
& \mathrel{\phantom{=}} + \im{ \int_0^t \int \overline{\Delta Iu} [(I\nabla u)\cdot \nabla u F''(u) - I(\nabla u \cdot \nabla u F''(u))]dx ds} \label{almost estimate 5} \\
&\mathrel{\phantom{=}} + \im{\int_0^t \int \overline{IF(u)} (F(Iu)-IF(u)) dx ds} \label{almost estimate 6}.
\end{align}
Let us consider $(\ref{almost estimate 1})$. By H\"older's inequality, we estimate
\begin{align}
|(\ref{almost estimate 1})| &\lesssim \|\Delta I u\|^2_{L^4_t L^{\frac{2d}{d-2}}_x} \|F'(Iu)-F'(u)\|_{L^2_t L^{\frac{d}{2}}_x} \nonumber \\
& \lesssim Z_I^2 \| |Iu-u|(|Iu|+|u|)^{\frac{8}{d}-1} \|_{L^2_t L^{\frac{d}{2}}_x} \nonumber \\
& \lesssim Z_I^2 \|P_{>N} u\|_{L^{\frac{16}{d}}_t L^4_x} \|u\|^{\frac{8}{d}-1}_{L^{\frac{16}{d}}_t L^4_x}. \label{almost estimate 1 sub 1}
\end{align}
By $(\ref{property 2})$, we bound
\begin{align}
\|P_{>N} u\|_{L^{\frac{16}{d}}_t L^4_x} \lesssim N^{-2} \|\Delta I u\|_{L^{\frac{16}{d}}_t L^4_x} \lesssim N^{-2} Z_I, \label{almost estimate 1 sub 2}
\end{align}
where $\Big(\frac{16}{d}, 4 \Big)$ is biharmonic admissible. Similarly, we have from $(\ref{property 3})$ that
\begin{align}
\|u\|_{L^{\frac{16}{d}}_t L^4_x} \lesssim Z_I. \label{almost estimate 1 sub 3}
\end{align}
Combining $(\ref{almost estimate 1 sub 1})-(\ref{almost estimate 1 sub 3})$, we get
\begin{align}
|(\ref{almost estimate 1})| \lesssim N^{-2} Z_I^{2+\frac{8}{d}}. \label{almost estimate 1 final}
\end{align}
We next bound
\begin{align}
|(\ref{almost estimate 2})| &\lesssim \|\Delta I u\|_{L^4_t L^{\frac{2d}{d-2}}_x} \||\nabla I u|^2\|_{L^{\frac{16}{11}}_t L^{\frac{4d}{4d-11}}_x} \|F''(Iu)-F''(u)\|_{L^{16}_t L^{\frac{4d}{15-2d}}_x} \nonumber \\
& \lesssim \|\Delta I u\|_{L^4_t L^{\frac{2d}{d-2}}_x} \|\nabla I u\|^2_{L^{\frac{32}{11}}_t L^{\frac{8d}{4d-11}}_x} \|F''(Iu)-F''(u)\|_{L^{16}_t L^{\frac{4d}{15-2d}}_x} \nonumber \\
& \lesssim Z_I^3 \||Iu-u|^{\frac{8}{d}-1} \|_{L^{16}_t L^{\frac{4d}{15-2d}}_x } \nonumber \\
&\lesssim Z_I^3 \|P_{>N} u\|^{\frac{8}{d}-1}_{L^{\frac{16(8-d)}{d}}_t L^{\frac{4(8-d)}{15-2d}}_x} \nonumber\\
&\lesssim N^{-2\left(\frac{8}{d}-1\right)} Z_I^{2+\frac{8}{d}}. \label{almost estimate 2 final}
\end{align}
Here we drop the $I$-operator and apply $(\ref{property 3})$ with the fact $\gamma>1$ to get the third line. We also use the fact that for $5\leq d \leq 7$,
\[
|F''(z)-F''(\zeta)| \lesssim |z-\zeta|^{\frac{8}{d}-1}, \quad \forall z, \zeta \in \C.
\] 
The last estimate uses $(\ref{almost estimate 1 sub 2})$. Note that $\Big(\frac{32}{11}, \frac{8d}{4d-11}\Big)$ and $\Big(\frac{16(8-d)}{d}, \frac{4(8-d)}{15-2d}\Big)$ are biharmonic admissible. Similarly, we estimate
\begin{align*}
|(\ref{almost estimate 3})| &\lesssim \|\Delta I u\|_{L^4_t L^{\frac{2d}{d-2}}_x} \|\nabla I u\|_{L^{\frac{32}{11}}_t L^{\frac{8d}{4d-11}}_x} \|\nabla I u- \nabla u\|_{L^{\frac{32}{11}}_t L^{\frac{8d}{4d-11}}_x}  \|F''(u)\|_{L^{16}_t L^{\frac{4d}{15-2d}}_x} \nonumber \\
&\lesssim Z_I^2 \|\nabla P_{>N} u\|_{L^{\frac{32}{11}}_t L^{\frac{8d}{4d-11}}_x} \|F''(u)\|_{L^{16}_t L^{\frac{4d}{15-2d}}_x}.
\end{align*}
We next use $(\ref{property 2})$ to have
\[
\|\nabla P_{>N} u\|_{L^{\frac{32}{11}}_t L^{\frac{8d}{4d-11}}_x} \lesssim N^{-1} \|\Delta I u\|_{L^{\frac{32}{11}}_t L^{\frac{8d}{4d-11}}_x} \lesssim N^{-1} Z_I.
\]
As $F''(u)=O(|u|^{\frac{8}{d}-1})$, we use $(\ref{property 3})$ to get
\begin{align}
\|F''(u)\|_{L^{16}_t L^{\frac{4d}{15-2d}}_x} \lesssim \|u\|^{\frac{8}{d}-1}_{L^{\frac{16(8-d)}{d}}_t L^{\frac{4(8-d)}{15-2d}}_x} \lesssim Z_I^{\frac{8}{d}-1}. \label{estimate second derivative}
\end{align}
We thus obtain 
\begin{align}
|(\ref{almost estimate 3})| \lesssim N^{-1} Z_I^{2+\frac{8}{d}}. \label{almost estimate 3 final}
\end{align}
By H\"older's inequality,
\[
|(\ref{almost estimate 4})| \lesssim \|\Delta I u\|_{L^2_t L^{\frac{2d}{d-4}}_x} \|(\Delta I u) F'(u) - I(\Delta u F'(u))\|_{L^2_t L^{\frac{2d}{d+4}}_x}.
\]
We then apply Lemma $\ref{lem rougher estimate}$ with $q=\frac{2d}{d+4}, q_1 = \frac{2d(d-3)}{d^2-7d+16}$ and $q_2=\frac{d(d-3)}{2(2d-7)}$ to get
\[
\|(\Delta I u) F'(u) - I(\Delta u F'(u))\|_{L^{\frac{2d}{d+4}}_x} \lesssim N^{-\alpha} \|\Delta I u\|_{L^{\frac{2d(d-3)}{d^2-7d+16}}_x} \|\scal{\nabla}^\alpha F'(u)\|_{L^{\frac{d(d-3)}{2(2d-7)}}_x},
\]
where $\alpha=2-\gamma+\delta$. The H\"older inequality then implies
\[
\|(\Delta I u) F'(u) - I(\Delta u F'(u))\|_{L^2_t L^{\frac{2d}{d+4}}_x} \lesssim N^{-\alpha} \|\Delta I u\|_{L^{\frac{2(d-3)}{d-4}}_t L^{\frac{2d(d-3)}{d^2-7d+16}}_x} \|\scal{\nabla}^\alpha F'(u)\|_{L^{2(d-3)}_t L^{\frac{d(d-3)}{2(2d-7)}}_x}.
\]
We have from $(\ref{estimate commutator}), (\ref{estimate 2}), (\ref{estimate 3})$ and $(\ref{estimate 4})$ that
\[
\|\scal{\nabla}^\alpha F'(u)\|_{L^{2(d-3)}_t L^{\frac{d(d-3)}{2(2d-7)}}_x} \lesssim Z_I^{\frac{8}{d}}.
\]
Thus
\begin{align}
|(\ref{almost estimate 4})| \lesssim N^{-(2-\gamma+\delta)} Z_I^{2+\frac{8}{d}}. \label{almost estimate 4 final}
\end{align}
Similarly, we bound
\begin{align}
|(\ref{almost estimate 5})| &\lesssim \|\Delta I u\|_{L^4_t L^{\frac{2d}{d-2}}_x} \|(I\nabla u)\cdot \nabla u F''(u) - I(\nabla u \cdot \nabla u F''(u))\|_{L^{\frac{4}{3}}_t L^{\frac{2d}{d+2}}_x}. \label{almost estimate 5 sub 1}
\end{align}
Applying Lemma $\ref{lem rougher estimate}$ with $q=\frac{2d}{d+2}, q_1 = \frac{8d}{4d-11}$ and $q_2=\frac{8d}{19}$ and using H\"older inequality, we have
\begin{align}
\|(I\nabla u)\cdot \nabla u F''(u) - I(\nabla u \cdot \nabla u F''(u))\|_{L^{\frac{4}{3}}_t L^{\frac{2d}{d+2}}_x} \lesssim N^{-\alpha} \|I\nabla u\|_{L^{\frac{32}{11}}_t L^{\frac{8d}{4d-11}}_x} \|\scal{\nabla}^\alpha (\nabla u F''(u)) \|_{L^{\frac{8}{5}}_t L^{\frac{8d}{19}}_x}. \label{almost estimate 5 sub 2}
\end{align}
The fractional chain rule implies
\begin{multline}
\|\scal{\nabla}^\alpha (\nabla u F''(u)) \|_{L^{\frac{8}{5}}_t L^{\frac{8d}{19}}_x} \lesssim \|\scal{\nabla}^{\alpha+1} u\|_{L^{\frac{32}{11}}_t L^{\frac{8d}{4d-11}}_x} \|F''(u)\|_{L^{16}_t L^{\frac{4d}{15-2d}}_x} \\
+ \|\nabla u\|_{L^{\frac{32}{11}}_t L^{\frac{8d}{4d-11}}_x} \|\scal{\nabla}^\alpha F''(u)\|_{L^{16}_t L^{\frac{4d}{15-2d}}_x}. \label{almost estimate 5 sub 3}
\end{multline}
By our assumptions on $\gamma$ and $\delta$, we see that $\alpha+1<\gamma$. By $(\ref{property 3})$ (and dropping the $I$-operator if necessary) and $(\ref{estimate second derivative})$, 
\begin{align}
\|I\nabla u\|_{L^{\frac{32}{11}}_t L^{\frac{8d}{4d-11}}_x}, \|\nabla u\|_{L^{\frac{32}{11}}_t L^{\frac{8d}{4d-11}}_x}, \|\scal{\nabla}^{\alpha+1} u\|_{L^{\frac{32}{11}}_t L^{\frac{8d}{4d-11}}_x} \lesssim Z_I, \quad  \|F''(u)\|_{L^{16}_t L^{\frac{4d}{15-2d}}_x} \lesssim Z_I^{\frac{8}{d}-1}. \label{almost estimte 5 sub 4}
\end{align}
Here $\left(\frac{32}{11}, \frac{8d}{4d-11}\right)$ is biharmonic admissible. It remains to bound $\|\scal{\nabla}^\alpha F''(u)\|_{L^{16}_t L^{\frac{4d}{15-2d}}_x}$. To do so, we use
\begin{align}
\|\scal{\nabla}^\alpha F''(u)\|_{L^{16}_t L^{\frac{4d}{15-2d}}_x} \lesssim \|F''(u)\|_{L^{16}_t L^{\frac{4d}{15-2d}}_x} + \||\nabla|^\alpha F''(u)\|_{L^{16}_t L^{\frac{4d}{15-2d}}_x}. \label{almost estimate 5 sub 5}
\end{align}
The first term in the right hand side is treated in $(\ref{estimate second derivative})$. For the second term in the right hand side, we make use of the fractional chain rule given in Lemma $\ref{lem fractional chain rule holder}$ with $\beta=\frac{8}{d}-1$, $\alpha=2-\gamma+\delta$, $q=\frac{4d}{15-2d}$ and $q_1, q_2$ satisfying
\[
\Big(\frac{8}{d}-1 -\frac{\alpha}{\rho}\Big) q_1 = \frac{\alpha}{\rho}q_2=\frac{4(8-d)}{15-2d},
\]
and $\frac{\alpha}{\frac{8}{d}-1}<\rho<1$. Note that the choice of $\rho$ is possible since $\alpha <\frac{8}{d}-1$ by our assumptions. With these choices, we have
\[
\Big(1-\frac{\alpha}{\beta \rho}\Big) q_1 = \frac{4d}{15-2d}>1,
\]
for $5\leq d\leq 7$. Then,
\begin{align*}
\||\nabla|^\alpha F''(u)\|_{L^{\frac{4d}{15-2d}}_x} \lesssim \||u|^{\frac{8}{d}-1-\frac{\alpha}{\rho}} \|_{L^{q_1}_x} \||\nabla|^\rho u\|^{\frac{\alpha}{\rho}}_{L^{\frac{\alpha}{\rho} q_2}_x} \lesssim \|u\|^{\frac{8}{d}-1-\frac{\alpha}{\rho}}_{L^{\left(\frac{8}{d}-1-\frac{\alpha}{\rho}\right)q_1}_x} \||\nabla|^\rho u\|^{\frac{\alpha}{\rho}}_{L^{\frac{\alpha}{\rho} q_2}_x}.
\end{align*}
By H\"older's inequality,
\begin{align*}
\||\nabla|^\alpha F''(u)\|_{L^{16}_tL^{\frac{4d}{15-2d}}_x} &\lesssim  \|u\|^{\frac{8}{d}-1-\frac{\alpha}{\rho}}_{L^{\left(\frac{8}{d}-1-\frac{\alpha}{\rho}\right)p_1}_tL^{\left(\frac{8}{d}-1-\frac{\alpha}{\rho}\right)q_1}_x} \||\nabla|^\rho u\|^{\frac{\alpha}{\rho}}_{L^{\frac{\alpha}{\rho} p_2}_tL^{\frac{\alpha}{\rho} q_2}_x} \\
&=\|u\|^{\frac{8}{d}-1-\frac{\alpha}{\rho}}_{L^{\frac{16(8-d)}{d}}_tL^{\frac{4(8-d)}{15-2d}}_x} \||\nabla|^\rho u\|^{\frac{\alpha}{\rho}}_{L^{\frac{16(8-d)}{d}}_t L^{\frac{4(8-d)}{15-2d}}_x},
\end{align*}
provided
\[
\Big(\frac{8}{d}-1 -\frac{\alpha}{\rho}\Big) p_1 = \frac{\alpha}{\rho}p_2=\frac{16(8-d)}{d}.
\]
Since $\left(\frac{16(8-d)}{d}, \frac{4(8-d)}{15-2d}\right)$ is biharmonic admissible, we have from $(\ref{property 3})$ with the fact $0<\rho<1<\gamma$ that
\begin{align}
\||\nabla|^\alpha F''(u)\|_{L^{16}_tL^{\frac{4d}{15-2d}}_x} \lesssim Z_I^{\frac{8}{d}-1}. \label{almost estimate 5 sub 6}
\end{align}
Collecting from $(\ref{almost estimate 5 sub 1})$ to $(\ref{almost estimate 5 sub 6})$, we get
\begin{align}
|(\ref{almost estimate 5})| \lesssim N^{-(2-\gamma+\delta)} Z_I^{2+\frac{8}{d}}. \label{almost estimate 5 final}
\end{align}
Finally, we consider $(\ref{almost estimate 6})$. We bound
\begin{align}
|(\ref{almost estimate 6})| &\lesssim \||\nabla|^{-1} IF(u)\|_{L^2_tL^{\frac{2d}{d-2}}_x} \|\nabla(F(Iu)-IF(u))\|_{L^2_t L^{\frac{2d}{d+2}}_x} \nonumber \\
&\lesssim \|\nabla I F(u)\|_{L^2_t L^{\frac{2d}{d+2}}_x} \|\nabla(F(Iu)-IF(u))\|_{L^2_t L^{\frac{2d}{d+2}}_x}.  \label{almost estimate 6 sub 1}
\end{align}
By $(\ref{commutator estimate 3})$, 
\[
\|\nabla I F(u)\|_{L^2_t L^{\frac{2d}{d+2}}_x} \lesssim Z_I^{1+\frac{8}{d}}.
\]
By the triangle inequality, we estimate
\[
\|\nabla(F(Iu)-IF(u))\|_{L^2_t L^{\frac{2d}{d+2}}_x} \lesssim \|(\nabla I u) (F'(Iu)-F'(u))\|_{L^2_t L^{\frac{2d}{d+2}}_x} + \|(\nabla I u) F'(u)- \nabla I F(u)\|_{L^2_t L^{\frac{2d}{d+2}}_x}.
\]
We firstly use H\"older's inequality and  estimate as in $(\ref{almost estimate 1 sub 1})$ to get
\begin{align}
\|(\nabla I u)(F'(Iu)-F'(u))\|_{L^2_t L^{\frac{2d}{d+2}}_x} &\lesssim \|\nabla I u\|_{L^\infty_t L^{\frac{2d}{d-2}}_x} \|F'(Iu)-F'(u)\|_{L^2_t L^{\frac{d}{2}}_x} \nonumber \\
&\lesssim \|\Delta I u\|_{L^\infty_t L^2_x} \|P_{>N} u\|_{L^{\frac{16}{d}}_t L^4_x} \|u\|^{\frac{8}{d}-1}_{L^{\frac{16}{d}}_t L^4_x} \nonumber \\
&\lesssim N^{-2} Z_I^{1+\frac{8}{d}}. \label{almost estimate 6 sub 2}
\end{align}
By $(\ref{commutator estimate 1})$, 
\begin{align}
\|(\nabla I u) F'(u)- \nabla I F(u)\|_{L^2_t L^{\frac{2d}{d+2}}_x} \lesssim N^{-(2-\gamma+\delta)} Z_I^{1+\frac{8}{d}}. \label{almost estimate 6 sub 3}
\end{align}
Combining $(\ref{almost estimate 6 sub 1})-(\ref{almost estimate 6 sub 3})$, we get
\begin{align}
|(\ref{almost estimate 6})| \lesssim Z_I^{1+\frac{8}{d}}(N^{-2} Z_I^{1+\frac{8}{d}} + N^{-(2-\gamma+\delta)} Z_I^{1+\frac{8}{d}}). \label{almost estimate 6 final}
\end{align}
The desired estimate $(\ref{almost conservation law})$ follows from $(\ref{almost estimate 1 final}), (\ref{almost estimate 2 final}), (\ref{almost estimate 3 final}), (\ref{almost estimate 4 final}), (\ref{almost estimate 6 final})$ and $(\ref{control size Z_I})$. The proof is complete.
\end{proof}
\section{Global well-posedness} \label{section global well-posedness}
\setcounter{equation}{0}
Let us now show the global existence given in Theorem $\ref{theorem global existence}$. By density argument, we assume that $u_0 \in C^\infty_0(\R^d)$. Let $u$ be a global solution to (NL4S) with initial data $u_0$. 
In order to apply the almost conservation law, we need the modified energy of initial data to be small. Since $E(Iu_0)$ is not necessarily small, we will use the scaling $(\ref{scaling})$ to make $E(Iu_\lambda(0))$ small. We have
\begin{align}
E(Iu_\lambda(0))=\frac{1}{2}\|Iu_\lambda (0)\|^2_{\dot{H}^2_x} +\frac{d}{2d+8} \|Iu_\lambda (0)\|^{\frac{2d+8}{d}}_{L^{\frac{2d+8}{d}}_x}. \label{modified energy}
\end{align}
We use $(\ref{property 5})$ to estimate
\begin{align}
\|Iu_\lambda(0)\|_{\dot{H}^2_x} \lesssim N^{2-\gamma}\|u_\lambda(0)\|_{\dot{H}^\gamma_x}= N^{2-\gamma} \lambda^{-\gamma} \|u_0\|_{\dot{H}^\gamma_x}. \label{homogeneous norm}
\end{align}
In order to make $\|Iu_\lambda(0)\|_{\dot{H}^2_x} \leq \frac{1}{8}$, we choose
\begin{align}
\lambda \approx N^{\frac{2-\gamma}{\gamma}}. \label{choice of lambda}
\end{align}
We next bound $\|I u_\lambda(0)\|_{L^{\frac{2d+8}{d}}_x}$. Note that we can easily estimate this norm by the Sobolev embedding  
\[
\|I u_\lambda(0)\|_{L^{\frac{2d+8}{d}}_x} \lesssim \|u_\lambda(0)\|_{L^{\frac{2d+8}{d}}_x} =\lambda^{-\frac{2d}{d+4}}\|u_0\|_{L^{\frac{2d+8}{d}}_x} \lesssim \lambda^{-\frac{2d}{d+4}} \|u_0\|_{H^\gamma_x}, 
\] 
but it requires $\gamma \geq \frac{2d}{d+4}$. In order to remove this requirement, we use the technique of \cite{I-teamglobal} (see also \cite{MiaoWuZhang}). We firstly separate the frequency space into the domains 
\[
\Omega_1:=\Big\{\xi \in \R^d, \  |\xi| \lesssim \frac{1}{\lambda}\Big\}, \quad \Omega_2:=\Big\{ \xi \in \R^d, \ \frac{1}{\lambda} \lesssim |\xi| \lesssim N \Big\}, \quad \Omega_3:= \Big\{ \xi\in \R^d, \ |\xi| \gtrsim N \Big\},
\]
and then write
\[
[u_\lambda(0)]\widehat{\ }(\xi)=(\chi_1(\xi) + \chi_2(\xi) + \chi_3(\xi)) [u_\lambda(0)]\widehat{\ } (\xi),
\] 
for non-negative smooth functions $\chi_j$ supported in $\Omega_j, j=1,2,3$ respectively and satisfying $\sum \chi_j(\xi)=1$. Thus
\[
I u_\lambda(0) = \chi_1(D) I u_\lambda(0) + \chi_2(D) I u_\lambda(0) + \chi_3(D) I u_\lambda(0).
\]
We now use the Sobolev embedding to have
\[
\|\chi_1(D)I u_\lambda(0)\|_{L^{\frac{2d+8}{d}}_x} \lesssim \||\nabla|^{\frac{2d}{d+4}} \chi_1(D) I u_\lambda(0)\|_{L^2_x} \lesssim \||\nabla|^{\frac{2d}{d+4}} \chi_1(D) I\|_{L^2_x \rightarrow L^2_x}\|u_\lambda(0)\|_{L^2_x}.
\]
Thanks to the support of $\chi_1$, the functional calculus gives 
\begin{align}
\||\nabla|^{\frac{2d}{d+4}} \chi_1(D) I\|_{L^2_x \rightarrow L^2_x} \lesssim \||\xi|^{\frac{2d}{d+4}-\alpha} |\xi|^\alpha \chi_1(\xi)\|_{L^\infty_\xi} \lesssim \lambda^{\alpha-\frac{2d}{d+4}}, \label{norm estimate 1}
\end{align}
provided $0<\alpha <\frac{2d}{d+4}$. Similarly,
\[
\|\chi_3(D) I u_\lambda(0)\|_{L^{\frac{2d+8}{d}}_x} \lesssim \||\nabla|^{\frac{2d}{d+4}} \chi_3(D) I u_\lambda(0)\|_{L^2_x} \lesssim \||\nabla|^{\frac{2d}{d+4}-\gamma} \chi_3(D) I\|_{L^2_x \rightarrow L^2_x} \|u_\lambda(0)\|_{\dot{H}^\gamma_x}.
\]
A direct computation shows
\begin{align}
\|u_\lambda(0)\|_{\dot{H}^\gamma_x} = \lambda^{-\gamma}\|u_0\|_{\dot{H}^\gamma_x}. \label{H gamma norm}
\end{align}
Using the support of $\chi_3$, the functional calculus again gives
\begin{align}
\||\nabla|^{\frac{2d}{d+4}-\gamma} \chi_3(D) I\|_{L^2_x \rightarrow L^2_x} \lesssim \||\xi|^{\frac{2d}{d+4}-\gamma}\chi_3(\xi) (N|\xi|^{-1})^{2-\gamma}\|_{L^\infty_\xi} \lesssim N^{\frac{2d}{d+4}-\gamma}. \label{estimate chi 3}
\end{align}
To obtain this bound, we split into two cases. \newline
\indent When $\frac{2d}{d+4} \geq \gamma$, we simply bound
\[
\||\xi|^{\frac{2d}{d+4}-\gamma}\chi_3(\xi) (N|\xi|^{-1})^{2-\gamma}\|_{L^\infty_\xi} \lesssim 1 \lesssim N^{\frac{2d}{d+4}-\gamma}.
\]
\indent When $\gamma> \frac{2d}{d+4}$, we write
\[
\||\xi|^{\frac{2d}{d+4}-\gamma}\chi_3(\xi) (N|\xi|^{-1})^{2-\gamma}\|_{L^\infty_\xi}= N^{\frac{2d}{d+4}-\gamma} \|(N|\xi|^{-1})^{\gamma-\frac{2d}{d+4}} \chi_3(\xi) (N|\xi|^{-1})^{2-\gamma}\|_{L^\infty_\xi} \lesssim N^{\frac{2d}{d+4}-\gamma}.
\]
Combining $(\ref{H gamma norm})$ and $(\ref{estimate chi 3})$, we get
\begin{align}
\|\chi_3(D) I u_\lambda(0)\|_{L^{\frac{2d+8}{d}}_x} \lesssim N^{\frac{2d}{d+4}-\gamma}\lambda^{-\gamma}\|u_0\|_{\dot{H}^\gamma_x}. \label{norm estimate 2}
\end{align}
We treat the intermediate case as
\[
\|\chi_2(D) I u_\lambda(0)\|_{L^{\frac{2d+8}{d}}_x} \lesssim \||\nabla|^{\frac{2d}{d+4}-\gamma} \chi_2(D) I\|_{L^2_x \rightarrow L^2_x} \|u_\lambda(0)\|_{\dot{H}^\gamma_x}.
\] 
We have
\[
\||\nabla|^{\frac{2d}{d+4}-\gamma} \chi_2(D) I\|_{L^2_x \rightarrow L^2_x} \lesssim \|\xi|^{\frac{2d}{d+4}-\gamma} \chi_2(\xi)\|_{L^\infty_\xi}. 
\]
\indent When $\frac{2d}{d+4}\geq \gamma$, we bound
\[
\|\xi|^{\frac{2d}{d+4}-\gamma} \chi_2(\xi)\|_{L^\infty_\xi} \lesssim N^{\frac{2d}{d+4}-\gamma}.
\]
\indent When $\gamma>\frac{2d}{d+4}$, we write
\[
\|\xi|^{\frac{2d}{d+4}-\gamma} \chi_2(\xi)\|_{L^\infty_\xi} = \|\xi|^{\frac{2d}{d+4}-\gamma-\beta} |\xi|^\beta \chi_2(\xi)\|_{L^\infty_\xi} \lesssim \lambda^{\beta+\gamma-\frac{2d}{d+4}},
\]
provided $\frac{2d}{d+4}-\gamma <\beta <\frac{2d}{d+4}$. These estimates together with $(\ref{H gamma norm})$ yield
\begin{align}
\|\chi_2(D) I u_\lambda(0)\|_{L^{\frac{2d+8}{d}}_x} \lesssim \left\{ \begin{array}{l l}
N^{\frac{2d}{d+4}-\gamma} \lambda^{-\gamma} \|u_0\|_{\dot{H}^\gamma_x} & \text{ if } \frac{2d}{d+4} \geq \gamma, \\
\lambda^{\beta -\frac{2d}{d+4}} \|u_0\|_{\dot{H}^\gamma_x} & \text{ if } \gamma>\frac{2d}{d+4}.
\end{array} \right. \label{norm estimate 3}
\end{align}
Collecting $(\ref{norm estimate 1}), (\ref{norm estimate 2}), (\ref{norm estimate 3})$ and use $(\ref{choice of lambda})$, we obtain
\begin{align}
\|Iu_\lambda(0)\|_{L^{\frac{2d+8}{d}}_x} \lesssim (\lambda^{\alpha-\frac{2d}{d+4}} + \lambda^{\beta -\frac{2d}{d+4}} + \lambda^{-\frac{8\gamma}{(d+4)(2-\gamma)}}) \|u_0\|_{H^\gamma_x}, \label{inhomogeneous norm}
\end{align}
for some $0<\alpha<\frac{2d}{d+4}$ and $\frac{2d}{d+4}-\gamma <\beta<\frac{2d}{d+4}$. Therefore, it follows from $(\ref{modified energy}), (\ref{homogeneous norm}), (\ref{choice of lambda})$ and $(\ref{inhomogeneous norm})$ by taking $\lambda$ sufficiently large depending on $\|u_0\|_{H^\gamma_x}$ and $N$ (which will be chosen later and depends only on $\|u_0\|_{H^\gamma_x}$) that
\[
E(Iu_\lambda(0)) \leq \frac{1}{4}.
\] 
Now let $T$ be arbitrarily large. We define
\[
X:=\{ 0\leq t\leq \lambda^4T \ | \ \|u_\lambda\|_{M([0,t])} \leq K t^{\frac{d-4}{8(d-3)}} \},
\]
with $K$ a constant to be chosen later. Here $M(J)$ is given in $(\ref{interaction morawetz interpolation})$. We claim that $X=[0,\lambda^4 T]$. Assume by contradiction that it is not so. Since $\|u_\lambda\|_{M([0,t])}$ is a continuous function of time, there exists $T_0 \in [0,\lambda^4T]$ such that
\begin{align}
\|u_\lambda\|_{M([0,T_0])} &> K T_0^{\frac{d-4}{8(d-3)}}, \label{lower bound} \\
\|u_\lambda\|_{M([0,T_0])} &\leq 2K T_0^{\frac{d-4}{8(d-3)}}. \label{upper bound}
\end{align}
Using $(\ref{upper bound})$, we are able to split $[0,T_0]$ into subintervals $J_k, k=1,...,L$ in such a way that
\[
\|u_\lambda\|_{M(J_k)} \leq \mu,
\]
where $\mu$ is as in Proposition $\ref{prop almost conservation law}$. The number $L$ of possible subinterval must satisfy
\begin{align}
L \sim \Big(\frac{2K T_0^{\frac{d-4}{8(d-3)}}}{\mu}\Big)^{\frac{8(d-3)}{d}} \sim T_0^{\frac{d-4}{d}}. \label{condition of L (1)}
\end{align}
Next, thanks to Proposition $\ref{prop almost conservation law}$, we see that for $1<\gamma<2$ and any $0<\delta<\gamma-1$,
\[
\sup_{[0,T_0]} E(Iu_\lambda(t)) \lesssim E(Iu_\lambda(0)) + N^{-(2-\gamma+\delta)} L,
\]
for $\max\left\{3-\frac{8}{d},\frac{8}{d} \right\} <\gamma<2$ and $0<\delta<\gamma+\frac{8}{d}-3$. Since $E(Iu_\lambda(0)) \leq \frac{1}{4}$, we need 
\begin{align}
N^{-(2-\gamma+\delta)} L \ll \frac{1}{4} \label{condition of L (2)}
\end{align}
in order to guarantee 
\begin{align}
E(Iu_\lambda(t)) \leq 1, \label{small modified energy norm}
\end{align} 
for all $t\in [0,T_0]$. As $T_0 \leq \lambda^4 T$, we have from $(\ref{condition of L (1)})$ and $(\ref{condition of L (2)})$ ant the choice of $\lambda$ given in $(\ref{choice of lambda})$ that
\begin{align}
N^{-(2-\gamma+\delta)} N^{\frac{4(2-\gamma)(d-4)}{\gamma d}} T^{\frac{d-4}{d}} \ll \frac{1}{4}, \label{relation N and T}
\end{align}
or 
\begin{align}
\frac{4(2-\gamma)(d-4)}{\gamma d} < 2-\gamma+\delta, \label{global condition}
\end{align}
for $\max \left\{3-\frac{8}{d}, \frac{8}{d}\right\}<\gamma<2 $ and $0<\delta<\gamma+\frac{8}{d}-3$. Since $2-\gamma+\delta <\frac{8}{d}-1$, the condition $(\ref{global condition})$ is possible if we have
\[
\frac{4(2-\gamma)(d-4)}{\gamma d} < \frac{8}{d}-1.
\]
This implies $\gamma>\frac{8(d-4)}{3d-8}$. Thus 
\[
\gamma>\max\left\{3-\frac{8}{d}, \frac{8}{d}, \frac{8(d-4)}{3d-8} \right\}.
\]
Next, by $(\ref{interaction morawetz interpolation})$, 
\[
\|u_\lambda\|_{M([0,T_0])} \lesssim T_0^{\frac{d-4}{8(d-3)}} \|u_0\|_{L^2_x}^{\frac{1}{d-3}} \|u_\lambda\|^{\frac{d-4}{d-3}}_{L^\infty_t([0,T_0],\dot{H}^{\frac{1}{2}}_x)}.
\]
We use $(\ref{property 2})$ and the definition of the $I$-operator to estimate
\begin{align*}
\|u_\lambda(t)\|_{\dot{H}^{\frac{1}{2}}_x} &\leq \|P_{\leq N} u_\lambda(t)\|_{\dot{H}^{\frac{1}{2}}_x} + \|P_{>N}u_\lambda(t)\|_{\dot{H}^{\frac{1}{2}}_x} \\
&\lesssim \|P_{\leq N} u_\lambda(t)\|_{L^2_x}^{\frac{3}{4}} \|P_{\leq N} u_\lambda(t)\|^{\frac{1}{4}}_{\dot{H}^2_x} + N^{-\frac{3}{2}} \|I u_\lambda(t)\|_{\dot{H}^2_x} \\
&\lesssim \|u_0\|_{L^2_x}^{\frac{3}{4}} \|Iu_\lambda(t)\|^{\frac{1}{4}}_{\dot{H}^2_x} + N^{-\frac{3}{2}} \|I u_\lambda(t)\|_{\dot{H}^2_x}.
\end{align*}
Thus,
\begin{align}
\|u_\lambda\|_{M([0,T_0])} \lesssim T_0^{\frac{d-4}{8(d-3)}} \|u_0\|_{L^2_x}^{\frac{1}{d-3}} \sup_{[0,T_0]} \Big( \|u_0\|_{L^2_x}^{\frac{3}{4}} \|Iu_\lambda(t)\|^{\frac{1}{4}}_{\dot{H}^2_x} + N^{-\frac{3}{2}} \|I u_\lambda(t)\|_{\dot{H}^2_x} \Big)^{\frac{d-4}{d-3}}. \label{morawetz norm u lambda}
\end{align}
Since $\|Iu_\lambda(t)\|_{\dot{H}^2_x}\lesssim \sqrt{E(Iu_\lambda(t))}$, we obtain from $(\ref{small modified energy norm})$ and $(\ref{morawetz norm u lambda})$,
\[
\|u_\lambda\|_{M([0,T_0])} \leq C T_0^{\frac{d-4}{8(d-3)}},
\] 
for some constant $C>0$.  This contradicts with $(\ref{lower bound})$ for an appropriate choice of $K$. We get $X=[0, \lambda^4T]$ with $T$ arbitrarily large and
\begin{align}
E(Iu_\lambda(\lambda^4T)) \leq 1. \label{modified energy estimate}
\end{align}
Note that under the condition of $\gamma$, we see from $(\ref{relation N and T})$ that the choice of $N$ makes sense for arbitrarily large $T$. 
Now, by the conservation of mass and $(\ref{modified energy estimate})$, we bound
\begin{align*}
\|u(T)\|_{H^\gamma_x}& \lesssim \|u(T)\|_{L^2_x}+ \|u(T)\|_{\dot{H}^\gamma_x} \lesssim \|u_0\|_{L^2_x}+ \lambda^\gamma\|u_\lambda(\lambda^4T)\|_{\dot{H}^\gamma_x} \\
&\lesssim \|u_0\|_{L^2_x}+ \lambda^\gamma\|Iu_\lambda(\lambda^4T)\|_{H^2_x}  \\
&\lesssim \lambda^\gamma \lesssim N^{2-\gamma}\lesssim T^{\alpha(\gamma,d)},
\end{align*}
where $\alpha(\gamma,d)$ is a positive number that depends on $\gamma$ and $d$. This a priori bound gives the global existence in $H^\gamma$. The proof is now complete.
\section*{Acknowledgments}
The author would like to express his deep gratitude to Prof. Jean-Marc BOUCLET for the kind guidance and encouragement. 



\begin{thebibliography}{99}
\bibitem[BKS00]{Ben-ArtziKochSaut} {\bf M. Ben-Artzi, H. Koch, J. C. Saut} {\it Disperion estimates for fourth-order Schr\"odinger equations}, C.R.A.S., 330, S\'erie 1, 87-92 (2000).

\bibitem[BCD11]{BCDfourier} {\bf H. Bahouri, J. Y. Chemin, R. Danchin}, {\it Fourier analysis and non-linear partial differential equations}, A Series of Comprehensive Studies in Mathemati\text{s} 343, Springer (2011).

\bibitem[CW91]{ChristWeinstein} {\bf M. Christ, I. Weinstein}, {\it Dispersion of small amplitude solutions of the generalized Korteweg-de Vries equation}, J. Funct. Anal. 100, No. 1, 87-109 (1991).

\bibitem[CKSTT02]{I-teamalmost} {\bf J. Colliander, M. Keel, G. Staffilani, H. Takaoka, T. Tao}, {\it Almost conservation laws and global rough solutions to a nonlinear Schr\"odinger equation}, Math. Res. Lett. 9, 659-682 (2002).

\bibitem[CKSTT04]{I-teamglobal} {\bf J. Colliander, M. Keel, G. Staffilani, H. Takaoka, T. Tao}, {\it Global existence and scattering for rough solutions of a nonlinear Schr\"odinger equation on $\R^3$}, Comm. Pure Appl. Math. 57, 987-1014 (2004).

\bibitem[CM75]{CoifmanMeyer} {\bf R. Coifman, Y. Meyer}, {\it On commutators of singular integrals and bilinear singular integrals}, AMS 212, 315-331 (1975).

\bibitem[CM91]{CoifmanMeyer91} {\bf R. Coifman, Y. Meyer}, {\it Ondelettes and operateurs III, Operateurs multilineaires}, Actualites Mathematiques, Hermann, Paris (1991).

\bibitem[DPST07]{SilvaPavlovicStaffilaniTzirakis} {\bf D. De Silva, N. Pavlovic, G. Staffilani, N. Tzirakis}, {\it Global well-posedness for the $L^2$-critical nonlinear Schr\"odinger equation in higher dimensions}, Commun. Pure Appl. Anal. 6, No. 4, 1023-1041 (2007).

\bibitem[Din1]{Dinhfract} {\bf V. D. Dinh}, {\it Well-posedness of nonlinear fractional Schr\"odinger and wave equations in Sobolev spaces}, arXiv:1609.06181 (2016).

\bibitem[Din2]{Dinhfourt} {\bf V. D. Dinh}, {\it On well-posedness, regularity and ill-posedness for the nonlinear fourth-order Schr\"odinger equation}, arXiv:1703.00891 (2017).

\bibitem[Din3]{Dinhglobal} {\bf V. D. Dinh}, {\it Global well-posedness for a $L^2$-critical nonlinear higher-order Schr\"odinger equation}, arXiv:1703.00903 (2017).


\bibitem[Guo10]{Guo} {\bf C. Guo}, {\it Global existence of solutions for a fourth-order nonlinear Schr\"odinger equation in $n+1$ dimensions}, Nonlinear Anal. 73, 555-563 (2010).


\bibitem[HHW06]{HaoHsiaoWang06} {\bf C. Hao, L. Hsiao, B. Wang}, {\it Well-posedness for the fourth-order Schr\"odinger equations}, J. Math. Anal. Appl. 320, 246-265 (2006).

\bibitem[HHW07]{HaoHsiaoWang07} {\bf C. Hao, L. Hsiao, B. Wang}, {\it Well-posedness of the Cauchy problem for the fourth-order Schr\"odinger equations in
high dimensions}, J. Math. Anal. Appl. 328, 58-83 (2007).

\bibitem[HJ05]{HuoJia} {\bf Z. Huo, Y. Jia}, {\it The Cauchy problem for the fourth-order nonlinear Schr\"odinger equation related to the vortex filament}, J. Differential Equations 214, 1-35 (2005).

\bibitem[Kar96]{Karpman} {\bf V. I. Karpman}, {\it Stabilization of soliton instabilities by higher-order dispersion: Fourth order nonlinear Schr\"odinger-type equations}, Phys. Rev. E 53 (2), 1336-1339 (1996).

\bibitem[KS00]{KarpmanShagalov} {\bf V. I. Karpman, A.G Shagalov}, {\it Stability of soliton described by nonlinear Schr\"odinger-type equations with higher-order dispersion}, Phys. D 144, 194-210 (2000).

\bibitem[KPV93]{KenigPonceVega} {\bf C. E. Kenig, G. Ponce, L. Vega}, {\it Well-posedness and scattering results for the gereralized Korteveg-de Vries equation via the contraction principle}, Comm. Pure Appl. Math 46, 527-620 (1993).

\bibitem[MXZ09]{MiaoXuZhao09} {\bf C. Miao, G. Xu, L. Zhao}, {\it Global well-posedness and scattering for the defocusing energy critical nonlinear Schr\"odinger equations of fourth-order in the radial case}, J. Differ. Eqn. 246, 3715-3749 (2009).

\bibitem[MXZ11]{MiaoXuZhao11} {\bf C. Miao, G. Xu, L. Zhao}, {\it Global well-posedness and scattering for the defocusing energy critical nonlinear Schr\"odinger equations of fourth-order in dimensions $d\geq 9$}, J. Differ. Eqn. 251, 3381-3402 (2011).

\bibitem[MWZ15]{MiaoWuZhang} {\bf C. Miao, H. Wu, J. Zhang}, {\it Scattering theory below energy for the cubic fourth-order Schr\"odinger equation}, Math. Nachr. 288, No. 7, 798-823 (2015).

\bibitem[MZ07]{MiaoZhang} {\bf C. Miao, B. Zhang}, {\it Global well-posedness of the Cauchy problem for nonlinear Schr\"odinger-type equations}, Discrete Contin. Dyn. Syst. 17, No. 1, 181-200 (2007).

\bibitem[Pau1]{Pausader} {\bf B. Pausader}, {\it Global well-posedness for energy critical fourth-order Schr\"odinger equations in the radial case}, Dynamics of PDE 4, No. 3, 197-225 (2007).

\bibitem[Pau2]{Pausadercubic} {\bf B. Pausader}, {\it The cubic fourth-order Schr\"odinger equation}, J. Funct. Anal. 256, 2473-2517 (2009).

\bibitem[PS10]{PausaderShao10} {\bf B. Pausader, S. Shao}, {\it The mass-critical fourth-order Schr\"odinger equation in higher dimensions}, J. Hyper. Differential Equations 7, No. 4, 651-705 (2010).

\bibitem[Tao06]{Tao} {\bf T. Tao}, {\it Nonlinear dispersive equations: local and global analysis}, CBMS Regional Conference Series in Mathemati\text{s} 106, AMS (2006).

\bibitem[VZ09]{VisanZhang} {\bf M. Visan, X. Zhang}, {\it Global well-posedness and scattering for a class of nonlinear Schr\"odinger equations below the energy space}, Differ. Integral Eqn. 22, 99-124 (2009).

\bibitem[Vis06]{Visanthesis} {\bf M. Visan}, {\it The focusing energy-critical nonlinear Schr\"odinger equation in dimensions five and higher}, PhD Thesis, UCLA (2006).


\end{thebibliography}
\end{document}